\documentclass[a4paper,notitlepage,11pt]{amsart}

\usepackage[english]{babel}
\usepackage[utf8x]{inputenc}
\usepackage[T1]{fontenc}

\usepackage[a4paper,top=3cm,bottom=2cm,left=3cm,right=3cm,marginparwidth=1.75cm]{geometry}

\usepackage{amsmath, amsfonts, amssymb, amsthm}\usepackage{graphicx}
\usepackage[colorinlistoftodos]{todonotes}
\usepackage[colorlinks=true, allcolors=blue]{hyperref}
\usepackage[all,cmtip]{xy}
\hypersetup{colorlinks,%
citecolor=black,%
filecolor=black,%
linkcolor=black,%
urlcolor=black}
\usepackage{epigraph}
\theoremstyle{definition}
\newtheorem{rmk}{Remark}[section]
\newtheorem{rmks}[rmk]{Remarks}
\newtheorem{defin}[rmk]{Definition}
\newtheorem{defs}[rmk]{Definitions}

\theoremstyle{plain}
\newtheorem{thm}[rmk]{Theorem}
\newtheorem{prop}[rmk]{Proposition}
\newtheorem{lem}[rmk]{Lemma}

\theoremstyle{remark}

\newcommand{\Z}{\mathbb{Z}}
\newcommand{\N}{\mathbf{N}} 
\newcommand{\R}{\mathbb{R}} 
\newcommand{\C}{\mathbb{C}}

\newcommand{\Q}{\mathbb{Q}}
\newcommand{\Gr}{\operatorname{Gr}}

\renewcommand{\P}{\mathbb{P}}

\newcommand{\M}{\mathbf{M}}
\newcommand{\Stab}{\operatorname{Stab}}
\renewcommand{\O}{\mathcal{O}}
\renewcommand{\S}{\mathbf{S}}

\newcommand{\NS}{\operatorname{NS}}
\newcommand{\Amp}{\operatorname{Amp}}

\renewcommand{\Im}{\operatorname{Im}} 
\renewcommand{\Re}{\operatorname{Re}} 
\renewcommand{\emptyset}{\varnothing} 
\newcommand{\tensor}{\otimes} 

\newcommand{\Aut}{\operatorname{Aut}}

\newcommand{\Hom}{\operatorname{Hom}}

\newcommand{\ch}{\operatorname{ch}}
\newcommand{\Coh}{\operatorname{Coh}}
\newcommand{\Ext}{\operatorname{Ext}}
\newcommand{\rk}{\operatorname{rk}}

\title{Bridgeland wall-crossing for some Simpson moduli spaces on $\P^1\times\P^1$}
\author{Matteo Altavilla}
\address{Department of Mathematics, University of Utah, 155 South 1400 East, Salt
Lake City, UT 84112-0090, USA}
\email{altavilla@math.utah.edu}

\begin{document}
\maketitle

\begin{abstract}
We investigate the wall-crossing behavior as Bridgeland moduli spaces for some Simpson moduli spaces of Gieseker-semistable torsion sheaves on $\P^1\times \P^1$ with linear Hilbert polynomial. In particular, we recover some of the birational transformations of these spaces described in \cite{maican1}, \cite{moon2} and \cite{moon1} as wall-crossing maps in certain slices of $\Stab(\P^1\times \P^1)$.

\end{abstract}
\tableofcontents

\section{Introduction}
Bridgeland first introduced stability conditions on triangulated categories in \cite{bridgeland1} as a tool to approach Douglas' theory of $D$-branes, and then applied it to the study of a smooth projective varieties via the bounded derived category of coherent sheaves $D^b(X)$, starting with $K3$ surfaces in \cite{bridgeland2}. The main result of this theory is the existence of a complex manifold $\Stab(X)$ parametrizing all such stability conditions, which has several nice features such as a wall-and-chamber structure, group actions, covering properties.

This field has grown quickly in the past decade, even though the very existence of such stability conditions for a specific variety $X$ is already something that needs a lot of effort to be proven. For curves all is known (see \cite{macri1}), while for surfaces a great deal of general machinery has been developed for the study of the structure of $\Stab(X)$ (see for example \cite{bridgeland2}, or \cite{macri2} for a survey); the most recent results in this direction have been achieved in \cite{abelian3}, \cite{li} and \cite{fano3}, with the construction of Bridgeland stability conditions for abelian and Fano threefolds, the ultimate goal being the non-emptiness of the stability manifold for Calabi-Yau threefolds.

Another important aspect of this theory is the study of moduli spaces of semistable objects $\M_\sigma(v)$ of a given class $v$ as $\sigma$ varies in $\Stab(X)$. First of all, when the coarse moduli spaces exist, it is not always true that they are projective, even if they always carry a nef line bundle (\cite{projk3}); projectivity has been proven in several cases like $\P^2$ in \cite{p2}, $K3$'s in \cite{projk3}, Enriques surfaces in \cite{enriques}, del Pezzo surfaces of Picard rank 1 in \cite{arcara}, but it remains still open in general, since the known proofs usually rely on special techniques like Fourier-Mukai transforms or quiver representations.

Once one gets a hold on some particular moduli spaces, the next question is the study of its birational geometry via wall-crossing: we know that Bridgeland moduli spaces for a fixed class of objects in the derived category are constant inside a chamber of $\Stab(X)$, and it undergoes some birational surgery when crossing a wall. Several known birational transformations between moduli spaces have been realized as Bridgeland wall-crossings; the first prominent result in this sense was the one in \cite{p2}, where the authors proved that the entire Minimal Model Program (MMP) for the Hilbert scheme of points on $\P^2$ could be realized by subsequent wall crossings along a ray in $\Stab(\P^2)$.

After that, the same result was proven for the moduli space of any class on $\P^2$ in \cite{zhaop2}, for $K3$'s in \cite{projk3}, for Enriques surfaces in \cite{enriques} and partially for the Hilbert scheme of points on $\P^1\times\P^1$ and the first del Pezzo surface in \cite{coskunbertram}.

Even when the wall-crossing does not yield the entire MMP, we can still recover crucial information about the birational geometry of the moduli spaces from it, especially when we are able to identify one of them with know projective varieties such as Gieseker moduli spaces. Problems in this sense have been studied for example in \cite{ben1} and \cite{ben2} for the moduli spaces of ideal sheaves of twisted cubics and elliptic quartic curves on $\P^3$.

In the present paper we address the case of certain Simpson moduli spaces of Gieseker-semistable torsion sheaves on $\P^1\times\P^1$, for which a lot of birational geometry is known from the work in \cite{maican1}, \cite{moon2} or \cite{moon1}; we proceed as follows: in Section 2, we recall the basic results from the theory of Bridgeland stability conditions on surfaces and we present some preliminary definitions and computations for the case in exam; in Section 3 we go over the wall-crossing behavior for three different moduli spaces of torsion sheaves, recovering most of the birational results from \cite{maican1}, \cite{moon2} and \cite{moon1}; finally, in the Appendix we carry out some calculations needed in the Theorems from Section 3.

\subsection*{Acknowledgements} The author would like to thank Aaron Bertram for introducing him to the subject and for many useful discussions.

\section{Preliminaries}
\subsection{Preliminaries on Bridgeland stability on surfaces}
Throughout this section $X$ will be a smooth projective surface, even if some of the first results hold in more generality for higher dimensional smooth projective varieties. We omit most of the proofs for the sake of simplicity, and we refer to \cite{macri2} for a complete and exhaustive treatment of the subject whenever a reference is not indicated explicitly.

We will denote by $D^b(X)$ the bounded derived category of coherent sheaves on $X$ and by $K_0(X)=K_0(\Coh(X))$ the Grothendieck group of $X$, equipped with the group  homomorphism $\ch \colon K_0(X) \to H^{2*}(X,\Q)$, i.e. the Chern character.
\begin{defin} A \emph{slicing} $\mathcal{P}$ of $D^b(X)$ is a collection of subcategories $\mathcal{P}(\phi)\subset D^b(X)$ for $\phi\in \R$ such that
\begin{itemize}
\item $\mathcal{P}(\phi)[1]=\mathcal{P}(\phi+1)$;
\item $\Hom (\mathcal{P}(\phi_1),\mathcal{P}(\phi_2))=0$ for $\phi_1>\phi_2$;
\item for all $E\in D^b(X)$ there exists a unique \emph{Harder-Narasimhan} filtration, i.e. there exist $\phi_1>\dots> \phi_m \in \R$, objects $E_i\in D^b(X)$ and triangles
\[
\xymatrix{0=E_0 \ar[r] & E_1\ar[d] \ar[r] &E_2 \ar[d]\ar[r] &\dots \ar[r] &E_{m-1} \ar[d]\ar[r]&E_m=E \ar[d]
\\ & A_1 \ar@{-->}[ul]& A_2 \ar@{-->}[ul]& & A_{m-1} \ar@{-->}[ul]& A_m \ar@{-->}[ul]}
\]
such that $A_i\in \mathcal{P}(\phi_i)$ for all $i=1,\dots,m$.
\end{itemize}

We then define the \emph{heart} of the slicing to be $\mathcal{P}(0,1]$, i.e.  the extension-closure of the set of subcategories $\{\mathcal{P}(\phi) \ | \ \phi\in(0,1]\}$.
\end{defin}
\begin{rmk} Giving a slicing of $D^b(X)$ is equivalent to giving a full abelian subcategory $\mathcal{A}\subset D^b(X)$ that is also the \emph{heart} of a bounded $t$-structure; the correspondence is obtained exactly by setting $\mathcal{A}:=\mathcal{P}(0,1]$.
\end{rmk}
\begin{defin} A \emph{stability function} on an abelian category $\mathcal{A}$ is an additive homomorphism $Z\colon K_0(\mathcal{A}) \to \C$ such that for all $0\neq E\in \mathcal{A}$ we have 
\[
\Im Z(E) \geq 0 \qquad \text{and} \qquad \Re Z(E) <0 \text{ if } \Im Z(E)=0.
\]
\end{defin}
\begin{rmk} We will always have $\mathcal{A} \subset D^b(X)$, and one usually asks $Z$ to factor through the Chern character; if we let $\ch(K_0(\mathcal{A})):=\Lambda$ (which is independent on the choice of $\mathcal{A}$ in this case), we can actually assume $Z \colon \Lambda\to \C$.
\end{rmk}
We then define 
\[\mu (E) := -\dfrac{\Re Z(E)}{\Im Z (E)}\]
to be the \emph{slope function} associated to the stability function $Z$, where we interpret dividing by $0$ to be equal to $+\infty$.

\begin{defin} An object $E \in \mathcal{A}$ is called \emph{(semi)stable} with respect to $Z$ if for all proper subobjects $F \subset E$ in $\mathcal{A}$ one has $\mu(F) < (\leq) \ \mu(E)$; equivalently, if for all quotients $E \twoheadrightarrow Q$
in $\mathcal{A}$ one has $\mu(E) < (\leq)\ \mu(Q)$.
\end{defin}

\begin{defin} A \emph{Bridgeland stability condition} $\sigma=(\mathcal{P}, Z)$ on $X$ is given by the datum of:
\begin{itemize}
\item A slicing $\mathcal{P}$ of $D^b(X)$;
\item A stability function $Z$ on its heart, called \emph{central charge};
\end{itemize}
with the following two properties:
\begin{itemize}
\item[(i)] $Z$ must be compatible with $\mathcal{P}$, i.e. for all $0\neq E \in \mathcal{P}(\phi)$ we must have $Z(E) \in \R_{>0}\cdot e^{i\pi\phi}$.
\item[(ii)] $Z$ has the \emph{support property}; there are several formulations of this condition, but for our purposes we'll use the following one: there exists a quadratic form $Q$ on $\Lambda$ that is negative definite on $\ker Z$ and such that given $E$ a semistable object, $Q(\ch(E),\ch(E))\geq0$.
\end{itemize}
\end{defin}
\begin{rmks} As before, a stability condition can also be thought as a pair $\sigma=(\mathcal{A},Z)$ with the same properties, but now $\mathcal{A}$ is the heart of a $t$-structure. Also, by definition we have that an object $E\in \mathcal{P}(\phi)$ is semistable, and we call it a \emph{semistable object of phase $\phi$}. From the compatibility condition we can also reformulate the Harder-Narasimhan filtration in terms of semistable objects and their phases.
\end{rmks}

We let $\Stab(X)$ be the set of Bridgeland stability conditions on $X$.
In \cite{bridgeland1}, Bridgeland introduced this kind of stability conditions and proved the following fundamental Theorem
\begin{thm} \label{bridgeland} There exists a metric on $\Stab(X)$ such that the forgetful map
\[
\mathcal{Z}\colon \xymatrix@R-2pc{\Stab(X) \ar[r]& \Hom(\Lambda, \C) \\ \sigma(\mathcal{P},Z) \ar@{|->}[r] & Z}
\]
is a local homeomorphism. In particular, $\Stab(X)$ is a complex manifold.
\end{thm}

We now give a description of the wall and chamber structure of $\Stab(X)$.
\begin{defin} Let $v$ and $w$ be two linearly independent vectors in $\Lambda$. A \emph{(numerical) wall} $\mathcal{W}_{v,w}$ for $v$ with respect to $w$ is the set
\[
\{\sigma=(\mathcal{P},Z) \in \Stab(X) \ | \ \Im Z(v) \Re Z(w) = \Im Z(w) \Re Z(v) \};
\]
by Theorem \ref{bridgeland}, a wall is a real submanifold of $\Stab(X)$ of real codimension 1.
\end{defin}

Not all numerical walls are relevant to the wall and chamber structure:
\begin{prop} \label{propactualwalls} Let $v \in \Lambda$ and $S\subset D^b(X)$ be any subset of objects of class $v$; then there exists a collection of walls $\{W^S_{v,w}\}$ with the following properties:
\begin{itemize}
\item[(i)] The collection $\{W^S_{v,w}\}$ is locally finite;
\item[(ii)] For every $\sigma=(\mathcal{P},Z)$ lying on some $W^S_{v,w}$, there exists a phase $\phi$ and an exact sequence $0\to F \to E \to Q \to 0$ of semistable objects in $\mathcal{P}(\phi)$ with $\ch(F)=w$ and $E\in S$;
\item[(iii)] If $\sigma_1$ and $\sigma_2$ lie in the same connected component of the complement of $\cup W^S_{v,w}$, then $E \in S$ is $\sigma_1$-stable if and only if it is $\sigma_2$-stable.
\end{itemize}
In particular, the stability of a fixed object is an open condition in $\Stab(X)$
\end{prop}

\begin{defs} Walls of this kind are called \emph{actual} walls for the set $S$. A \emph{chamber} is a connected component of the complement of the set of actual walls.
\end{defs}

\begin{rmk} \label{remarkwalls} The moduli space $\M_\sigma(v)$ of $S$-equivalence classes of Bridgeland (semi)stable objects of class $v$ with respect to a stability condition $\sigma$ is constant inside a chamber and when the coarse moduli exists it is always proper and separated (see \cite{macri2}). Letting $\M_\sigma(v)$ varying along a path in $\Stab(X)$ that crosses a wall at some point $\sigma_0$ corresponds to a birational surgery, and the two moduli spaces on $\M^+$ and $\M^-$ on either side of the wall con be seen as projective bundles over the space $\M^0$ at the wall in a small neighborhood of the wall; indeed, we can take the objects $F$ and $Q$ given by Proposition \ref{propactualwalls} (ii) to be stable at the wall (it's enough for one of the two to be), so that they will be stable in a small open set and the semistable objects in $\M^{\pm}$ after the wall crossing are obtained by taking extensions of $F$ and $Q$ in the other direction.
\end{rmk}

Now we turn to our attention to some results that only hold for surfaces, but first we need a few definitions.

\begin{defs} Let $B \in \NS(X)$ be the class of a divisor; the \emph{$B$-twisted Chern character} of an object $E$ is formally defined as $\ch^B(E):= \ch(E)\cdot e^{-B}$. In particular, on surfaces we have
\begin{align*}
\ch_0^B&(E)=\ch_0(E), \qquad \; \ch_1^B(E)= \ch_1(E)-B\ch_0(E), \\ &\ch_2^B(E)= \ch_2(E)-B\ch_1(E)+\dfrac{B^2}{2}\ch_0(E).
\end{align*}
Let also $\omega$ be an ample class; the associated \emph{$B$-twisted Mumford slope} (when defined) is given by
\[
\mu_{\omega,B}(E):= \dfrac{\omega\ch_1^B(E)}{\omega^2\ch_0^B(E)}.
\]
\end{defs}

For fixed $\omega$ and $B$ we define two full subcategories of $\Coh(X)$ as follows:
\begin{align*}
&\mathcal{T}^{\omega,\beta}=\{E \in \Coh(X) \ | \ E \text{ is torsion or }\forall \ E \twoheadrightarrow Q\neq 0, \ \mu_{\omega,B}(Q)>0\} \\
&\mathcal{F}^{\omega,\beta}=\{E \in \Coh(X) \ | \ \forall \ 0\neq F \subset E, \  \mu_{\omega,B} \leq 0\}.
\end{align*}
Using properties of the classical Mumford stability one easily checks that $(\mathcal{F}^{\omega,\beta}, \mathcal{T}^{\omega,\beta})$ is a \emph{torsion pair} for $\Coh(X)$, i.e. $\Hom(\mathcal{T}^{\omega,\beta}, \mathcal{F}^{\omega,\beta}) =0$ and for all objects $E \in \Coh(X)$ there exists a unique short exact sequence
\[
0 \to T \to E \to F \to 0
\]
with $T\in \mathcal{T}^{\omega,\beta}$ and $F \in \mathcal{F}^{\omega,\beta}$.

We can then construct the extension-closure category
\[
\Coh^{\omega,B}(X):= \langle \mathcal{F}^{\omega,\beta}[1], \mathcal{T}^{\omega,\beta}\rangle,
\]
called the \emph{tilt with respect to the torsion pair}, that has also the following equivalent description
\[
\Coh^{\omega,B}(X)=\{ E \in D^b(X) \ | \ H^i(E)=0 \text{ for }i\neq 0,-1, \; H^{-1}(E)\in \mathcal{F}^{\omega,\beta}, \; H^0(E)\in \mathcal{T}^{\omega,\beta} \}.
\]
Using general results about tilting and torsion pairs from \cite{happelsmalo} we get the following Proposition:
\begin{prop} The category $\Coh^{\omega,B}(X)$ is the heart of a bounded $t$-structure in $D^b(X)$; in particular, $\Coh^{\omega,B}(X)$ is abelian.
\end{prop}
Now we explicitly define a stability function on the tilted heart as
\[
Z_{\omega,B}(E) := -\left(\ch_2^B(E)-\dfrac{\omega^2}{2}\ch_0^B(E)\right) + i \omega \ch_1^B(E),
\]
with corresponding slope function
\[
\nu_{\omega,B}(E) = \dfrac{\ch_2^B(E)-\dfrac{\omega^2}{2}\ch_0^B(E)}{\omega \ch_1^B(E)}.
\]
This leads us to the main Theorem regarding surfaces
\begin{thm} Let $X$ be a smooth projective surface; for all $\omega$ and $B$ the pair $\sigma_{\omega,B}= (\Coh^{\omega,B},Z_{\omega,B})$ defines a stability condition on $X$. Moreover the map
\[ 
\xymatrix@R-2pc{\Amp(X) \times \NS(X) \ar[r]& \Stab(X) \\ (\omega,B) \ar@{|->}[r] & \sigma_{\omega,B}}
\]
is a continuous embedding.
\end{thm}

The support property for $\sigma_{\omega,B}$ is given by the following quadratic form
\[
\Delta_{\omega,B}(E) = (\omega\ch_1^B(E))^2 - 2\omega^2\ch_0^B(E)\ch_2^B(E),
\]
which is easily proven to be independent of $B$ by expansion, so we can always take $B=0$ when computing it; hence we have an inequality for semistable objects, called the \emph{Bogomolov inequality}:

\begin{prop} \label{bogomolov} Let $E$ be a semistable object with respect to the stability condition $\sigma_{\omega, B}$; then $\Delta_{\omega,B}(E)=\Delta_{\omega}(E)\geq 0$.
\end{prop}

From now on we fix some ample divisor $H$ and we focus on the family of stability conditions of the kind $\sigma_{\omega,B}$ given by $\omega=\alpha H$ for $\alpha>0$ and $B=\beta H$ for $\beta\in \R$, to which we will refer as the \emph{$(\alpha,\beta)$-plane} (or \emph{slice}) corresponding to $H$; we will also use the subscript $(\alpha,\beta)$ in place of $(\omega, B)$ wherever necessary, and introduce the following compact notation which that be useful in computations:
\[
H\cdot \ch^\beta(F) = (H^2\ch_0^\beta(F), \ H\ch_1^\beta(F),\ \ch_2^\beta(F)).
\]
Notice also that by definition the category $\Coh^{\alpha,\beta}(X)$ is independent of $\alpha$, so that we will refer to it as just $\Coh^\beta(X)$; also one has $\Delta_{\alpha H}(E)=\alpha^2\Delta_{H}(E)$, meaning that $\alpha$ does not affect the Bogomolov Inequality either.

An immediate consequence of the definition of $\Coh^\beta(X)$ using Proposition \ref{propactualwalls} is the following Lemma on the numerics of the objects involved in an actual wall, which we will use repeatedly in the next section:

\begin{lem} \label{bounds} Let $0\to F \to E \to Q \to 0$ be a sequence defining a wall in $\Coh^\beta(X)$ for an object $E$ with $H\ch_1^\beta(E)>0$; then $0 < H\ch_1^\beta(F) < H\ch_1^\beta(E)$ and $0 < H\ch_1^\beta(Q) < H\ch_1^\beta(E)$.
\end{lem}
\begin{proof} By definition of $\Coh^\beta(X)$ we have $H\ch_1^\beta(F)\geq 0$ and $H\ch_1^\beta(Q)\geq 0$, which taken together with the linearity of the Chern character on exact sequences give the other side of the inequality for both objects as well: $H\ch_1^\beta(F) \leq H\ch_1^\beta(E)$ and $H\ch_1^\beta(Q) \leq H\ch_1^\beta(E)$. Now, if $H\ch_1^\beta(F)=0$ then $\nu_{\alpha,\beta}(F)=+\infty$ for all $\alpha$ and $F$ cannot define a wall for $E$; similarly, $H\ch_1^\beta(Q) \neq 0$, and again by difference we get the strict inequality on the other side as well.
\end{proof}

Another result that will be useful later is the following Proposition about walls for objects that are honest sheaves:

\begin{prop} \label{subofsheaf} If  $0\to F \xrightarrow{\phi} E \to Q\to 0$ is a short exact sequence in $\Coh^\beta(X)$ and $E$ is a sheaf in $\mathcal{T}^\beta$, then $F$ is also a sheaf in $\mathcal{T}^\beta$ and $\phi$ is a map of sheaves, but not necessarily injective.
\end{prop}
\begin{proof} From the long exact sequence in cohomology we get
\[
0\to H^{-1}(F) \to 0 \to H^{-1}(Q) \to H^0(F)\xrightarrow{\psi} E \to H^0(Q) \to 0,
\]
which yields $H^{-1}(F)=0$; this implies by definition that $F$ is a sheaf and it is in $\mathcal{T}^\beta$, and moreover $\phi$ coincides with $\psi$. This also shows that if  $H^{-1}(Q)\neq 0$ we have $\phi$ not injective at the level of sheaves, which happens exactly when $F$ has bigger rank than $E$ (since $H^{-1}(Q)=\ker \psi$ must be torsion-free).
\end{proof}

We then have a very useful structure Theorem regarding the walls in the $(\alpha,\beta)$-plane for a fixed class $v\in \Lambda$:

\begin{thm} \label{structureofwalls} Let $X$ be a surface, and fix a slice of $\Stab(X)$ corresponding to some ample divisor $H$; moreover, fix a class $v\in \Lambda$. Then we have the following:
\begin{itemize} 
\item[(i)]The (numerical) walls for $v$ are either semicircles with center on the $\beta$-axis or vertical rays;
\item[(ii)] The walls are all disjoint;
\item[(iii)] The (possibly degenerate) hyperbola $\nu_{\alpha,\beta}(v)=0$ intersects all the semicircular walls in their top point (i.e. the point right above the center);
\item[(iv)] If $\ch_0(v)\neq 0$ then there is a unique vertical wall given by the equation
\[
\beta=\dfrac{H\ch_1(v)}{H^2\ch_0(v)},
\]
and on either side of it the walls are strictly nested semicircles;
\item[(v)] If $\ch_0(v)=0$ then there is no vertical wall, and the walls are all concentric semicircles;
\item[(vi)] If a numerical wall is an actual wall at a point then it is a wall at every point.
\end{itemize}
\end{thm}

We will apply Theorem \ref{structureofwalls} throughout the paper, alongside with two other key results; the first one is known as Bertram's Lemma for sheaves (or \emph{``sliding down the wall''} principle), and it holds thanks to the fact that walls in the $(\alpha,\beta)$-plane are all disjoint (see \cite{p2} for a proof).

\begin{lem} Let $E$ be an object in $\mathcal{T}^{\beta_0} \subset \Coh^{\beta_0}(X)$ for some $\beta_0\in \R$, and let $F \xrightarrow{\phi} E$ be an inclusion in $\Coh^\beta_0(X)$ at some point $(\alpha_0,\beta_0)$ belonging to a wall $\mathcal{W}$ for $E$ (so that by Proposition \ref{subofsheaf} also $F$ is a sheaf in $\mathcal{T}^{\beta_0}$). Then the map $F \xrightarrow{\phi} E$ is an inclusion of semistable objects of the same slope in $\Coh^{\beta}(X)$ for all $(\alpha,\beta)\in \mathcal{W}$.
\end{lem}

Finally, we state the so-called Large Volume Limit Theorem (for a proof of this result in the most general setting one can combine the proofs in \cite{macri2} and \cite{p2})

\begin{thm} \label{lvl} Fix $H$ an ample divisor and let $v$ be a class such that $\Delta_{H}(v)\geq 0$; then walls in the $(\alpha,\beta)$-plane for objects of class $v$ are finite and the moduli space in the outermost chamber is isomorphic to the moduli of Gieseker-semistable objects of the same class.
\end{thm}
\subsection{Preliminaries on $\P^1\times\P^1$ and torsion sheaves}
From here on, $X=\P^1\times \P^1$; let $D_1$ and $D_2$ be the divisors corresponding to the two rulings of the quadric, so that $D_1^2 = D_2^2=0$ and $D_1 \cdot D_2=1$. It is known that $D_1$ and $D_2$ are generators for the N\'{e}ron-Severi and we will use the notation $(a,b)$ for the divisor $aD_1+bD_2$; a divisor is ample if and only if $a, \ b>0$, and throughout the paper we will use either $H=(1,2)$ or $H=(1,1)$; the choice will depend on whether we need to distinguish between objects with similar numerics or not.

Indeed, for all the computations involving an object $E$ we will always consider the quantity $d= H \cdot \ch_1(E)$, so that objects with different first Chern classes may have the same value for $d$; we will call the triple $(\rk(E), H \cdot \ch_1(E), \ch_2(E))= (r,d,c)$ the \emph{invariants} of the object $E$. More explicitly, for an object with $\ch_1(E)=(a,b)$ we have $d=2a+b$ if $H=(1,2)$ and $d=a+b$ if $H=(1,1)$.
We will also use the classical notation $\O(a,b) = \pi_1^*\O(a)\tensor \pi_2^*\O(b)$, where $\pi_1$ and $\pi_2$ are the projections onto the two factors of $\P^1\times\P^1$. 

Notice that for any line bundle $L=\O(a,b)$ and any choice of $H$ the invariants of $L$ are all integers, since $r, a$ and $b$ are and $\ch_2(L)= \ch_1(L)^2/2 = ab$; in virtue of the fact that we have a resolution by line bundles for every sheaf on $\P^1\times \P^1$ (see for example \cite{kapranov}), we get that the invariants $(r,d,c)$ are all in $\Z$ for any sheaf, and therefore for all objects in $D^b(X)$.

Now let $F$ be an object of invariants $(r,d,c)$; by definition we have
\[
H \cdot \ch^\beta(F)=\left(H^2 r, d-H^2\beta r, c- \beta d+ H^2 \dfrac{\beta^2}{2}r\right),
\]
and
\[
\nu_{\alpha,\beta}(F) = \dfrac{c-\beta d+ H^2 \left(\dfrac{\beta^2}{2}- \dfrac{\alpha^2}{2}\right)r}{d-H^2\beta r}.
\]
In particular, if $E$ is an object of invariants $(0,d',c')$ with $d'\neq 0$ we get
\[
H \cdot \ch^\beta(E)=(0, d', c'- \beta d'),
\]
and
\[\nu_{\alpha,\beta}(E) = \dfrac{c'-\beta d'}{d'}= \dfrac{c'}{d'}- \beta.
\]

The equation of a wall of the type $\mathcal{W}_{(r,d,c),(0,d',c')}$ is therefore given by
\begin{equation}
H^2r(\beta^2 +\alpha^2) - 2\dfrac{c'}{d'}H^2r\beta +2\dfrac{c'd}{d'}-2c =0 \label{walleq},
\end{equation}
and we can assume $r\neq0$, otherwise the equation would be independent from $\alpha$ and $\beta$ and it would not give a wall; hence, the center and the radius for the wall are given by
\begin{equation}
\qquad C=\left(\dfrac{c'}{d'},0\right) \qquad \qquad R = \sqrt{\left(\dfrac{c'}{d'}\right)^2-\dfrac{2c'd}{H^2d'r}-\dfrac{2c}{H^2r}}\label{generalradius}.
\end{equation}
We are interested in this kind of walls because in Section 3 we will investigate wall-crossing behavior for moduli spaces of certain torsion sheaves on $\P^1\times\P^1$:
\begin{defin} Let $(a,b)$ be a divisor on $\P^1\times\P^1$ and let $p(m):=(a+b)m+\chi$ be a linear polynomial in $m$. We define $\M(0, (a,b), p(m))$ to be the \emph{Simpson moduli space} of Gieseker-semistable torsion sheaves $E$ on $\P^1\times\P^1$ with $\ch_1(E)=(a,b)$ and Hilbert polynomial equal to $p(m)$.
\end{defin}

These moduli spaces are all projective by \cite{simpson}; moreover, the projectivity of all the other moduli spaces in our construction is ensured by \cite{arcara}. We conclude this section with a technical result from \cite{arcara2} that will grant stability of at least one of the two destabilizing objects involved in each wall crossing for Section 3, in light of Remark \ref{remarkwalls}:
\begin{thm} Line bundles on $\P^1\times\P^1$ are stable for all stability conditions of type $\sigma_{\omega,B}$.
\end{thm}

\section{Wall crossing}
\subsection{Wall crossing for $\M(0,(2,3),5m+2)$}
Let $\M= \M(0,(2,3),5m+2)$; first of all, by \cite{simpson} and \cite{lepot} we have that $\M$ is a projective variety of dimension 13. Then we have the following Theorem from \cite{maican1} describing the strata of $\M$ in terms of minimal resolutions:
\begin{thm}\label{maicanstrata1}
There exists a decomposition of $\M$ into three disjoint strata $\M_0$, $\M_1$ and $\M_2$. The strata have the following properties: 
\begin{itemize} 
\item[(i)] $\M_0$ is open in $\M$ and it corresponds to the set of sheaves $\mathcal{E}$ with a resolution of the form
\begin{equation}
0\to\O(-1,-2)\oplus\O(-1,-1) \xrightarrow{\phi} \O^2 \to \mathcal{E} \to 0, \label{res1}
\end{equation}
having $\phi_{12}$ and $\phi_{22}$ linearly independent;
\item[(ii)] $\M_1$ is a closed subvariety of codimension 1 and it corresponds to the set of sheaves $\mathcal{E}$ with a resolution of the form
\begin{equation}
0\to\O(-1,-2)\oplus\O(-2,-1) \xrightarrow{\phi} \O(-1,-1)\oplus\O(0,1) \to \mathcal{E} \to 0,  \label{res2}
\end{equation}
having $\phi_{11}\neq 0$ and $\phi_{12}\neq 0$; equivalently, $\mathcal{E}$ fits into a short exact sequence
\[
0\to \O_C(0,1) \to \mathcal{E}\to \C_p \to 0,
\]
with $C\in |\O(2,3)|$ and $\C_p$ the skyscraper sheaf at a point $p\in C$;
\item[(iii)] $\M_2$ is a closed subvariety of codimension 2 isomorphic to $\P^{11}$, and $\mathcal{E}\in\M_2$ if and only if $\mathcal{E}\simeq \O_C(1,0)$ for $C\in |\O(2,3)|$. Equivalently, it corresponds to the set of sheaves $\mathcal{E}$ with resolution
\begin{equation}
0\to \O(-1,-3)\to \O(1,0) \to \mathcal{E}\to 0. \label{res3}
\end{equation}

\end{itemize}
\end{thm}
In this section we're going to use $H=(1,2)$, so that $H^2=4$; by Riemann-Roch, we have that for $\mathcal{E} \in \M$ the Chern character is given by $v=\ch \mathcal{E} = (0, (2,3),-3)$ and $H\cdot \ch^{\beta}\mathcal{E}=(0,7,-3-7\beta)$, so that by (\ref{generalradius}) the walls for $v$ in the $(\alpha,\beta)$-plane all have center and radius given by
\begin{equation}
C=\left(-\dfrac{3}{7},0\right) \qquad R=\sqrt{\dfrac{9}{49}+\dfrac{3d}{14r}+\dfrac{c}{2r}}. \label{radius} 
\end{equation}
By Theorem \ref{lvl} we also know that $\M$ appears as the moduli space $\M_\sigma(v)$ in the outermost chamber, after finitely many walls.

From the resolutions given in Theorem \ref{maicanstrata1}, we can see that the open stratum has a potential destabilizing subobject given by
\[
[0\to \O(-1,-1) \to \O^2],
\]
which is derived equivalent to the ideal sheaf $\mathcal{I}_{p,q}(1,1)$ for $p,q$ two points on the quadric, and has Chern character $(1,(1,1),-1)$; this means that the objects $\mathcal{I}_{p,q}(1,1)$ and $\O(-1,-2)[1]$ determine a wall $\mathcal{W}_0$ of radius $R_0= \sqrt{\dfrac{16}{49}}=\dfrac{4}{7}$ below which no sheaf in the open stratum can be stable. Therefore, $\mathcal{W}_0$ is a candidate to be a collapsing wall for this moduli space and indeed we'll prove in Theorem \ref{main1} that it is, and that this is the only wall at which sheaves in the open stratum are destabilized when coming down from the Gieseker chamber. For now, we restrict our attention to the possible numerical walls that have radius bigger than $R_0$.

\begin{prop}\label{numwalls1} There are only two possible numerical walls for the class $v=\ch E = (0,(2,3),-3)$ with radius bigger than $R_0$, when $H=(1,2)$. In fact, if a sequence $0\to F \to E \to Q \to 0$ defines one of such walls, the invariants for $F$ (resp. $Q$) can only be $(r,d,c) = (1,1,0)$ or $(1,2,0)$.
\end{prop}
\begin{proof}
From the general equation of a wall (\ref{walleq}) we know that $r\neq 0$, otherwise the equation becomes independent of $\alpha$ and $\beta$; moreover we can assume that $F$ has positive rank, since this argument is symmetric in $F$ and $Q$ and $r(F)+r(Q)=0$.

Let then $F$ be a destabilizing subobject for an object of class $v$ along a wall of radius $R>R_0$; by Bertram's Lemma, $F$ has to be a subobject at every point of the wall, and if $R>R_0$ then the wall must intersect the vertical lines $\beta = -3/7+4/7=1/7$ and $\beta=-3/7-4/7=-1$ for positive $\alpha$, so that in particular $F$ must be a subobject for those values of $\beta$.

By Lemma \ref{bounds}, we have
\[\dfrac{4}{7}r < d < 7 +\dfrac{4}{7}r \quad \text{and} \quad  -4r < d < 7 -4r;
\]
since $d$ must be an integer, the first inequality tells us that $d\geq 1$ for all $r\geq1$, while the second inequality gives $d\leq -1$ when $r\geq2$. This means we must have $r=1$. Now when $r=1$ the second inequality yields $d <3$, which only leaves us with $d=1$ or $d=2$.

The Bogomolov inequality for $F$ yields $d^2-8c\geq 0$, so that $c\leq \dfrac{d^2}{8}$; moreover, since the wall must have radius bigger that $R_0$, by (\ref{radius}) we also have that $\dfrac{3d}{14}+\dfrac{c}{2}>\dfrac{1}{7}$, which yields $c>\dfrac{2-3d}{7}$. Since $c$ must be an integer we have $c=0$ both for $d=1$ and $d=2$, which proves the claim.
\end{proof}
Let $\mathcal{W}_2$ and $\mathcal{W}_1$ be respectively the outermost and innermost wall from Proposition \ref{numwalls1}. Using similar numerical arguments we can also prove the following Lemma

\begin{lem}\label{mainlemma} Let $H=(1,2)$.
\begin{itemize}
\item[(a)]For all $\alpha>0$ and $\beta<\dfrac{1}{4}$, the only tilt-semistable object with invariants $(1,1,0)$ is the line bundle $\O(0,1)$;
\item[(b)] For all $\alpha>0$ and $\beta<\dfrac{1}{2}$, the only tilt-semistable objects with invariants $(1,2,0)$ are the line bundles $\O(1,0)$ and $\O(0,2)$.
\end{itemize}
\end{lem}
\begin{proof} \begin{itemize}
\item[(a)] By definition, $\beta=\dfrac{1}{4}$ is the vertical wall for any class of invariants $(r,d,c)=(1,1,0)$ when $H=(1,2)$, and the sheaf $\O(0,1)$ is in $\Coh^\beta(X)$ for $\beta<\dfrac{1}{4}$. Now let $E$ be an object of a certain class $w$ with those invariants: the equation of the hyperbola $\nu_{\alpha,\beta}(E)=0$ is
\[
2(\beta^2-\alpha^2)-\beta=0
\]
so that the left branch intersects the $\alpha$-axis at the origin; hence by Theorem \ref{structureofwalls} every numerical wall will have a point on it with $\beta=0$.
Therefore, for there to be a wall given by a subobject $F$, by Lemma \ref{bounds} we would have that $0<d(F)<1$: since $d$ must be an integer, we conclude that there are no walls and therefore $E$ is always stable or unstable on either side of the vertical wall. Now suppose $E$ is stable on the left side of the vertical wall: by the Large Volume Limit, $E$ must be a Gieseker (semi)stable sheaf of class $w$. 
Given those invariants, there exists $k$ an odd integer such that
\[w=\left(1,\left(\dfrac{1-k}{2},k\right),0\right), 
\]
so then $E'=E\tensor \O\left(-\dfrac{1-k}{2},-k\right)$ is a torsion-free sheaf such that 
\[
\ch(E')= \left(1,(0,0),-\dfrac{k(1-k)}{2}\right).
\]
By embedding $E'$ into its double dual $E'^{\vee\vee}\simeq \O$ we must have $-\dfrac{k(1-k)}{2} \leq 0$, and the only possibility is $k=1$ which yields $\ch(E')=(1,(0,0),0)$. Then $E'\simeq \O$ and $E\simeq \O(0,1)$.
\item[(b)] Again by definition, the vertical wall for all classes of invariants $(r,d,c)=(1,2,0)$ is given by $\beta=\dfrac{1}{2}$, and the sheaves $\O(1,0)$ and $\O(0,2)$ belong to $\Coh^\beta(X)$ for $\beta<\dfrac{1}{2}$. Let $E$ be an object of class $w$ with those invariants, then the equation of the hyperbola $\nu_{\alpha,\beta}(E)=0$ is given by
\[
\beta^2-\alpha^2-\beta=0
\]
while the equation of a potential wall for a subobject $F$ of invariants $(r',d',c')$ is 
\[
(d'-2r')(\beta^2-\alpha^2)-2c\beta+c=0,
\]
so that it has center at the point $\left(\dfrac{c'}{d'-r'},0\right)$. As before, the left branch of the hyperbola crosses the axis at the origin, so that all semicircular walls have a point  of coordinate $\beta=0$ thanks to which we get the inequality $0<d'<2$, hence $d'=1$; moreover, since the hyperbola crosses all the walls at their top point, we have that the centers all have negative $\beta$-coordinate.
Now assume we are looking for the largest wall on the left of the vertical wall: since by Large Volume Limit the moduli in the outermost chamber consists only of torsion-free sheaves, from the long exact sequence in cohomology the destabilizing object for the biggest wall can only have $r'>0$.

Putting everything together we get $c'\geq 0$, but then the Bogomolov Inequality for $F$ yields $1-8r'c'\geq 0$, which in turn implies $c'\leq 0$; this means that $c'=0$ and there is no largest wall, so there's no wall at all and the only tilt-semistable objects are torsion-free sheaves with the given invariants.

Now we just repeat the last part of the argument from $(a)$: given those invariants, there exists $k$ an even integer such that
\[\ch(E)=\left(1,\left(\dfrac{2-k}{2},k\right),0\right), 
\]
so then $E'=E\tensor \O\left(-\dfrac{2-k}{2},-k\right)$ is a torsion-free sheaf such that 
\[
\ch(E')= \left(1,(0,0),-\dfrac{k(2-k)}{2}\right).
\]
This time, this yields two possibilities: $k=0$ or $k=2$, so that $E\simeq \O(1,0)$ or $E\simeq \O(0,2)$.
\end{itemize}
\end{proof}

We are now ready to describe all the walls for $\M$; see also \ref{fig1} below.

\begin{thm}\label{main1} Let $v=(0, (2,3),-3)$ and $H=(1,2)$; there are three walls and four chambers for $\M_\sigma(v)$ on $\P^1\times\P^1$ in the $(\alpha,\beta)$-plane. From the outermost to the innermost, the walls are given by the following pairs of objects:
\begin{itemize}
\item[(i)] $\mathcal{W}_2: \O(1,0), \ \O(-1,-3)[1]$;
\item[(ii)] $\mathcal{W}_1: \O(0,1), \ [\O(-1,-2)\oplus\O(-2,-1) \to \O(-1,-1)]$;
\item[(iii)] $\mathcal{W}_0: \mathcal{I}_{p,q}(1,1), \ \O(-1,-2)[1]$.
\end{itemize}
In the same order, the moduli spaces corresponding to each chamber are given by
\begin{itemize}
\item[(i)] the Simpson moduli space $\M=\M(0, (2,3), 5m+2)$;
\item[(ii)] a projective variety $\M'$ obtained from $\M$ by contracting $\M_2$ and replacing it with $\M_2'\simeq \P^1$;
\item[(iii)] a projective variety $\M''$ isomorphic to a GIT quotient, obtained from $\M'$ by contracting a $\P^{10}$-bundle on $\P^1\times\P^1$ onto its base;
\item[(iv)] the empty set $\emptyset$.
\end{itemize}
\end{thm}
\begin{proof} By Theorem \ref{lvl} we know that there are finitely many walls and in the outermost chamber we have $\M_\sigma(v)=\M$; by Proposition \ref{numwalls1}, we also know that the first potential wall we encounter from the outside in corresponds to a short exact sequence of semistable objects $0\to F \to E \to Q\to 0$ with $E\in \M$ and one between $F$ and $Q$ of invariants $(1,2,0)$; by Proposition \ref{subofsheaf} $F$ is a sheaf so it must be $F$ to have those invariants.

Since by (\ref{radius}) $\mathcal{W}_2$ is all on the left of $\beta=\dfrac{1}{2}$, by Lemma \ref{mainlemma} we have that $F\simeq \O(1,0)$ or $F\simeq \O(0,2)$; now using the resolutions from Theorem \ref{maicanstrata1} we obtain $\Hom(\O(0,2),E)=0$ for all $E\in\M$, so that necessarily $F\simeq \O(1,0)$. Moreover, using the same resolutions we also see that $\Hom(\O(1,0),E)=0$ for $E\not \in \M_2$, while clearly $\O(1,0)$ is a subobject for $E \in \M_2$, with quotient $Q\simeq \O(-1,-3)[1]$ (see Appendix for all the computations). 

Therefore $\mathcal{W}_2$ is an actual wall for $\M_2$ and on the inside of this wall $\M_2$ is replaced by the subvariety $\M_2' = \P(\Ext^1(\O(1,0),\O(-1,-3)[1]) \simeq \P(\Ext^2(\O(1,0),\O(-1,-3)) \simeq \P(\Hom(\O(-1,-3),\O(-1,-2))\simeq \P^1$, so that we obtain $\M_\sigma(v)=\M'$.

According to Proposition \ref{numwalls1}, the next potential wall we encounter is $\mathcal{W}_1$; suppose the wall is given by $0\to F\to E \to Q \to 0$, then by Lemma \ref{mainlemma} one between $F$ and $Q$ must be $\O(0,1)$ since $\mathcal{W}_1$ entirely lies on the left of $\beta=\dfrac{1}{4}$. Now $E\in \M'$ is still a sheaf unless $E\in \M_2'$, so first we show that the stratum $\M_2'$ is not involved in this second wall-crossing. Indeed if $\O(0,1)$ was a quotient of $E\in \M_2'$, from the long exact sequence in cohomology we would get
\[
0\to H^{-1}(F) \to \O(-1,-3) \to 0 \to H^0(F) \to \O(1,0) \to \O(0,1) \to 0,
\]
and we get a contradiction from the fact that there is no map of sheaves between $\O(1,0)$ and $\O(0,1)$; analogously, if $\O(0,1)$ is a subobject we get
\[
0 \to \O(-1,-3) \to H^{-1}(Q) \to \O(0,1) \xrightarrow{\phi} \O(1,0) \to H^0(Q)\to 0,
\]
and again since $\phi=0$ we get a short exact sequence 
\[
0 \to \O(-1,-3) \to H^{-1}(Q) \to \O(0,1)\to 0,
\]
which yields a contradiction once we shift it by 1, since $\O(-1,-3)[1]$ and $H^{-1}(Q)[1]$ are in the heart by definition along $\mathcal{W}_1$, but $\O(0,1)[1]$ is not.

So we are left with the case where $E$ is a sheaf and $\O(0,1)$ is the destabilizing subobject; in a similar fashion as before (see Appendix) we can prove that $\Hom(\O(0,1),E)=0$ for $E\in \M_0$, so that the open stratum is not involved in this wall crossing either, leaving us with $\M_1$ which clearly has $\O(0,1)$ as a subobject according to (\ref{res2}). Hence $\mathcal{W}_1$ is an actual wall for all the objects in $\M_1$ and the quotient $Q$ is derived equivalent to the complex $[[\O(-1,-2)\oplus\O(-2,-1) \xrightarrow{\phi} \O(-1,-1)]$, with the only condition that both components of $\phi$ are non-zero. Now $\Hom(\O(-1,-2),\O(-1,-1))=\Hom(\O(-2,-1),\O(-1,-1))\simeq \C^2$, so that there is a $\P^1\times\P^1$-worth of possible quotients obtained by varying $\phi$, and one can use this resolution to compute $\Ext^1(Q,\O(0,1))\simeq\C^{11}$ (see Appendix), so that $\M_1$ is a $\P^{10}$-bundle over $\P^1\times\P^1$.
When crossing on the other side of $\mathcal{W}_1$, $\M_1$ is replaced by the space of extensions in the other direction; since $\Ext^1(\O(0,1),Q)\simeq \C$ (see Appendix), we have that the wall crossing just contracts the $\P^{10}$-bundle onto its base. The space $\M''$ obtained this way therefore has an open subset isomorphic to $\M_0$, and the complement is given by two disjoint components isomorphic to $\P^1$ and $\P^1\times\P^1$ respectively: by Proposition 4.1 and 4.2 in \cite{maican1}, $\M''$ is isomorphic to a GIT quotient $W/G$ of a certain subset $W \subset \Hom(\O(-1,-2)\oplus\O(-1,-1),\O^2)$ by the group 
\[
G= (\Aut(\O(-1,-2)\oplus\O(-1,-1))\times \Aut(\O^2))/ \C^*.
\]

Finally, we know again from Theorem \ref{maicanstrata1} that the sheaf $\mathcal{I}_{p,q}(1,1)$ is a destabilizing subobject for the open stratum at $\mathcal{W}_0$, with quotient $Q\simeq \O(-1,-2)[1]$; since one has $\Ext^1(\mathcal{I}_{p,q}(1,1),\O(-1,-2)[1])=0$ (see Appendix), we see that $\mathcal{W}_0$ is a collapsing wall as expected: in fact, when crossing this wall the open stratum vanishes and this means the other strata vanish as well, otherwise we would get a birational map from $\M''$ to a space of lower dimension. This proves that $\M_\sigma(v)$ is empty once we cross the wall $\mathcal{W}_0$ and it concludes the proof.
\end{proof}

\begin{rmk}
The destabilizing subobjects appearing in Theorem \ref{main1} can be recovered from Theorem \ref{maicanstrata1} as maximal linear strands in the resolutions for each strata. This is also true for the moduli spaces in the remaining sections.
\end{rmk}
\vspace{-0.3cm}
\begin{figure}[htbp]
\centering
\includegraphics[scale=0.19]{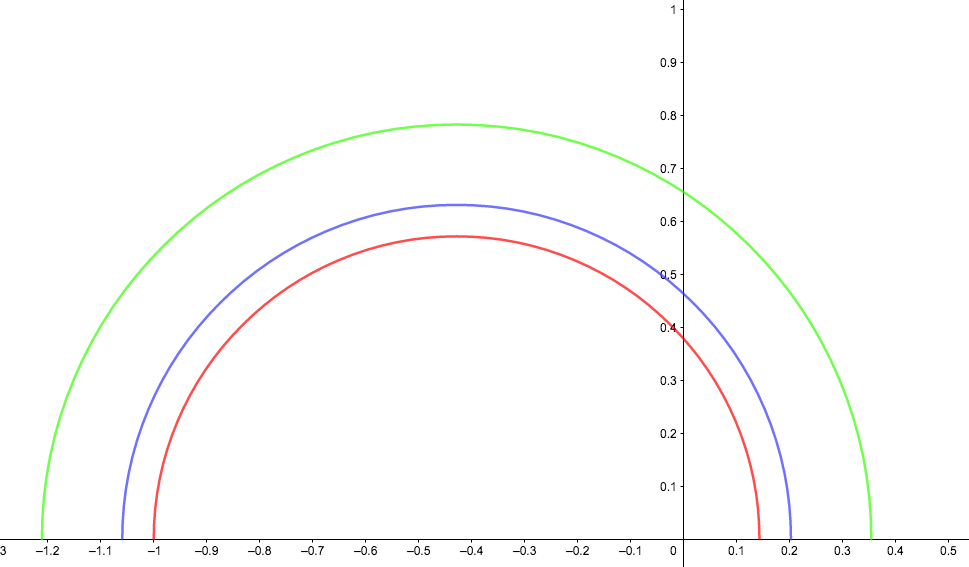}
\vspace{-0.25cm}
\caption{The walls for $\M(0, (2,3), 5m+1)$ \label{fig1}}
\end{figure}
\vspace{-0.3cm}
\subsection{Wall crossing for $\mathbf{M}(0,(2,3),5m+1)$}
Let $\N= \M(0,(2,3),5m+1)$; by \cite{simpson} and \cite{lepot} we have that $\N$ is a projective variety of dimension 13. As before, we have a structure Theorem from \cite{maican2} describing the strata of $\N$ in terms of minimal resolutions:

\begin{thm}\label{maicanstrata2}
There exists a decomposition of $\N$ into four strata $\N_0$, $\N_1$, $\N_2$ and $\N_3$. The strata have the following properties: 
\begin{itemize} 
\item[(i)] $\N_0$ is open in $\N$ and it corresponds to the set of sheaves $\mathcal{E}$ with a resolution of the form
\begin{equation}
0\to\O(-1,-2)^2 \xrightarrow{\phi} \O(0,-1)\oplus \O \to \mathcal{E} \to 0, \label{res1'}
\end{equation}
where $\phi_{12}$ and $\phi_{22}$ define a subscheme of length 2;
\item[(ii)] $\N_1$ is a closed subvariety of codimension 1 isomorphic to a $\P^9$-bundle over $\P^2\times\P^1$, and it corresponds to the set of sheaves $\mathcal{E}$ with a resolution of the form
\begin{equation}
0\to\O(-2,-1)\oplus\O(-1,-3) \xrightarrow{\phi} \O(-1,-1)\oplus\O \to \mathcal{E} \to 0,  \label{res2'}
\end{equation}
having $\phi_{11}\neq 0$ and $\phi_{12}\neq 0$;

\item[(iii)] $\N_2$ is a closed subvariety of codimension 2 isomorphic to $\P^{11}$, and $\mathcal{E}\in\N_2$ if and only if $\mathcal{E}\simeq \O_C(0,1)$ for $C\in |\O(2,3)|$. Equivalently, it corresponds exactly to the set of sheaves $\mathcal{E}$ with resolution
\begin{equation}
0\to \O(-2,-2)\to \O(0,1) \to \mathcal{E}\to 0. \label{res3'}
\end{equation}
\item[(iv)] $\N_3$ is a closed subvariety of codimension 3 isomorphic to a $\P^{1}$-bundle over $\P^8\times\P^1$, and $\mathcal{E}\in\N_3$ if and only if $\mathcal{E}$ fits in an extension 
\[
0\to \O_D \to \mathcal{E}\to \O_L \to 0,
\]
where $D\in|\O(2,2)|$ and $L\in|\O(0,1)|$. Equivalently, it corresponds to the set of sheaves $\mathcal{E}$ with resolution
\begin{equation}
0\to \O(-2,-2)\oplus \O(0,-1)\to \O^2 \to \mathcal{E}\to 0,\label{res4'}
\end{equation}
\end{itemize}
The strata are all disjoint except for $\N_2$and $\N_3$ which intersect along a subvariety isomorphic to $\P^8\times\P^1$ consisting of sheaves $\mathcal{E}\simeq \O_C(0,1)$ with $C=D \cup L$.
\end{thm}

In this section we will be using $H=(1,1)$ since we don't need to separate any walls, hence $H^2=2$; by Riemann-Roch, we have that for $\mathcal{E} \in \N$ the Chern character is given by $v=\ch \mathcal{E} = (0, (2,3),-4)$ and $H\cdot \ch^{\beta}\mathcal{E}=(0,5,-4-5\beta)$, so that by (\ref{generalradius}) the walls for $v$ in the $(\alpha,\beta)$-plane all have center and radius given by
\begin{equation}
C=\left(-\dfrac{4}{5},0\right) \qquad R=\sqrt{\dfrac{16}{25}+\dfrac{4d}{5r}+\dfrac{c}{r}}. \label{radius2} 
\end{equation}
By Theorem \ref{lvl} we also know that $\N$ appears as the moduli space $\M_\sigma(v)$ in the outermost chamber, after finitely many walls.

Here, the potential collapsing wall $\mathcal{W}'_0$ is given by the object $\O$, so that by (\ref{radius2}) it has radius $R_0=\dfrac{4}{5}$. The following Proposition deals with the potential walls that have radius bigger than $R_0$.

\begin{prop}\label{numwalls2}
There is only one possible numerical wall for the class $v=\ch E = (0,(2,3),-4)$ with radius bigger than $R_0$ when $H=(1,1)$. Indeed, if a sequence $0\to F \to E \to Q \to 0$ defines one of such walls we have that the invariants for $F$ (resp. $Q$) can only be $(r,d,c) = (1,1,0)$.
\end{prop}
\begin{proof}
The proof is similar to that of Proposition \ref{numwalls1}. As before, we can restrict to $r>0$; the two values for $\beta$ we get from $R>R_0$ are $\beta=-4/5+4/5=0$ and $\beta=-4/5-4/5=-8/5$; Bertram's Lemma and Lemma \ref{bounds} yield
\[0 < d < 5 \quad \text{and} \quad  -\dfrac{16}{5}r < d < 5 -\dfrac{16}{5}r,
\]
so that we get a contradiction for $r\geq 2$; if $r=1$ we can only have $d=1$.

Now if $r=1$ and $d=1$ the condition $R>R_0$ yields $c > -\dfrac{4}{5}$ and the Bogomolov Inequality for $F$ gives $1-4c\geq 0$, so that $c=0$.
\end{proof}
Let $\mathcal{W}'_1$ be the wall from the Proposition \ref{numwalls2}; we then have the analogue to Lemma \ref{mainlemma}
\begin{lem} \label{mainlemma2} Let $H=(1,1)$; for all $\alpha>0$ and $\beta<\dfrac{1}{2}$, the only tilt-semistable object with invariants $(1,1,0)$ are the line bundles $\O(0,1)$ and $\O(1,0)$.
\end{lem}
\begin{proof} The proof is the exact same one as Lemma \ref{mainlemma} $(a)$ with minor adjustments due to the different choice of $H$.
\end{proof}

We're now ready to prove the wall-crossing Theorem for $\N$; see also Figure \ref{fig2} below.
\begin{thm}\label{main2} Let $v=(0, (2,3),-4)$ and $H=(1,1)$; there are two walls and three chambers for $\M_\sigma(v)$ on $\P^1\times\P^1$ in the $(\alpha,\beta)$-plane. From the outermost to the innermost, the walls are given by the following pairs of objects:
\begin{itemize}
\item[(i)] $\mathcal{W}'_1: \O(0,1), \ \O(-2,-2)[1]$;
\item[(ii)] $\mathcal{W}'_0: \O, \ [\O(-1,-2)^2 \to \O(0,-1)]$;
\end{itemize}
In the same order, the moduli spaces corresponding to each chamber are given by
\begin{itemize}
\item[(i)] the Simpson moduli space $\N=\M(0, (2,3), 5m+1)$;
\item[(ii)] a projective variety $\N'$ isomorphic to a projective bundle over a blow-up of the Grassmannian $\Gr(2,4)$; $\N'$ is obtained from $\N$ by contracting the stratum $\N_2$ and replacing it with $\N'_2\simeq \P^1$;
\item[(iii)] the empty set $\emptyset$.
\end{itemize}
\end{thm}
\begin{proof} By Theorem \ref{lvl} we know there are only finitely many walls and in the outermost chamber we have $\M_\sigma(v)=\N$; by Proposition \ref{numwalls2} we know the largest possible wall is given by $\mathcal{W}'_1$, and since $\N$ contains only sheaves we know that the destabilizing subobject at that wall must be either $\O(1,0)$ or $\O(0,1)$ in virtue of Proposition \ref{subofsheaf} and Lemma \ref{mainlemma2}, being $\mathcal{W}'_1$ all on the left of $\beta=\dfrac{1}{2}$.

Using the resolutions from Theorem \ref{maicanstrata2} we can compute $\Hom(\O(1,0),E)=0$ for all $E\in\N$ (see Appendix), so that we must have $F\simeq \O(0,1)$; similarly $\Hom(\O(0,1),E)=0$ unless $E\in \N_2$, for which the quotient is $Q\simeq \O(-2,-2)[1]$. Hence $\mathcal{W}_1'$ is an actual wall for the stratum $\N_2$ and by crossing the wall it gets replaced by the projectivization of $\Ext^1(\O(0,1), \O(-2,-2)[1]) = \Hom(\O,\O(0,1))=\C^2$. Since $\N_3$ intersects $\N_2$ we have that this wall-crossing also affects $\N_3$, and indeed one can see $\O(0,1)$ being a subobject of (\ref{res3'}) as $[\O(0,-1) \to \O]$, but this can only happen when the extension splits and indeed $E\simeq \O_C(0,1)$; with this interpretation, the newly obtained $\P^1$ exactly replaces the zero-section of the $\P^1$-bundle over $\P^8\times\P^1$ in $\N_3$.

From \cite{moon2}, we know this operation gives a birational map from $\N$ to a new space $\N'$ which is a $\P^9$-bundle over the blow-up of $\Gr(2,4)$ along a $\P^1$ parametrizing lines in $\P^3$ of type $(1,0)$.

Finally, $\mathcal{W}'_0$ is a collapsing wall for $\N'$: in fact, by Theorem \ref{maicanstrata2} we know that $\O$ is a destabilizing subobject for all the sheaves in the open stratum and also $Q\simeq [\O(-1,-2)^2 \to \O(0,-1)]$ for $E\in \N_0$; once we cross the wall we have $\Ext^1(\O,Q)=0$  (see Appendix) and this is enough to say that the open stratum vanishes and we get the empty set.
\end{proof}
\vspace{-0.2cm}
\begin{figure}[htbp]
\centering
\includegraphics[scale=0.20]{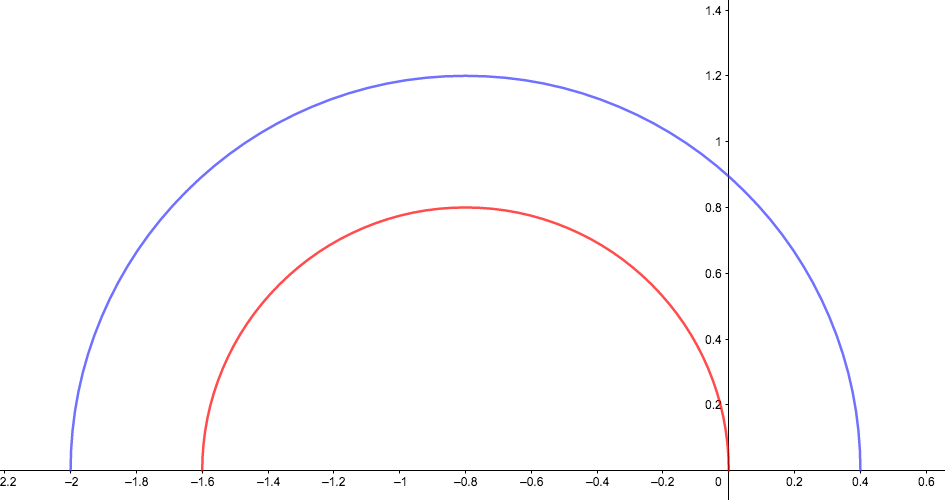}
\vspace{-0.2cm}
\caption{The walls for $\M(0, (2,3), 5m+2)$ \label{fig2}}
\end{figure}
\vspace{-0.2cm}
\subsection{Wall crossing for $\mathbf{M}(0,(2,2),4m+2)$}
This example is more straightforward than the previous two and it was already present in \cite{moon1} in some form; we report it here for the sake of completeness..
Let $\S= \M(0,(2,2),4m+2)$; by \cite{simpson} and \cite{lepot} we have that $\S$ is a projective variety of dimension 9. As before, we have a structure Theorem from \cite{moon1} describing the strata of $\S$ in terms of minimal resolutions:
\begin{thm}\label{maicanstrata3}
There exists a decomposition of $\S$ into three disjoint strata $\S_0$, $\S_1$ and $\S_2$. The strata have the following properties: 
\begin{itemize} 
\item[(i)] $\S_0$ is open in $\S$ and it corresponds to the set of sheaves $\mathcal{E}$ with a resolution of the form
\begin{equation}
0\to\O(-1,-1)^2 \to \O^2 \to \mathcal{E} \to 0; \label{res1''}
\end{equation}
\item[(ii)] $\S_1 \simeq \P^8$ is a divisor and it corresponds exactly to the set of sheaves $\mathcal{E}$ with a resolution
\begin{equation}
0\to\O(-2,-1)\to \O(0,1) \to \mathcal{E} \to 0;  \label{res2''}
\end{equation}
\item[(iii)] $\S_2\simeq\P^8$ is also a divisor and it corresponds exactly to the set of sheaves $\mathcal{E}$ with resolution
\begin{equation}
0\to \O(-1,-2)\to \O(1,0) \to \mathcal{E}\to 0. \label{res3''}
\end{equation}
\end{itemize}
\end{thm}
\begin{rmk} Notice that $\S$ in this case is singular, since the Chern class of its elements is not primitive; this can be detected also from this Theorem, by interpreting $\S$ as a moduli of semistable objects: the singular locus corresponds to the strictly semistable objects in the open stratum that have a subobject of the form $[\O(-1,-1) \to \O]$, whose slope is always equal to that of $\mathcal{E}$.
\end{rmk}

In this section we go back to using $H=(1,2)$; by Riemann-Roch, we have that for $\mathcal{E} \in \S$ the Chern character is given by $v=\ch \mathcal{E} = (0, (2,2),-2)$ and $H\cdot \ch^{\beta}\mathcal{E}=(0,6,-2-6\beta)$, so that by (\ref{generalradius}) the walls for $v$ in the $(\alpha,\beta)$-plane all have center and radius given by
\begin{equation}
C=\left(-\dfrac{1}{3},0\right) \qquad R=\sqrt{\dfrac{1}{9}+\dfrac{d}{6r}+\dfrac{c}{2r}}. \label{radius3} 
\end{equation}
In this example, the potential collapsing wall $\mathcal{W}''_0$ is given by the object $\O^2$, so that by (\ref{radius3}) it has radius $R_0=\dfrac{1}{3}$. The following Proposition deals with the potential walls that have radius bigger than $R_0$.
\begin{prop}\label{numwalls3}
There are only two possible numerical walls for the class $v=\ch E = (0,(2,2),-2)$ with radius bigger than $R_0$ when $H=(1,2)$. Indeed, if a sequence $0\to F \to E \to Q \to 0$ defines one of such walls we have that the invariants for $F$ (resp. $Q$) can only be $(r,d,c) = (1,1,0)$ or $(1,2,0)$.
\end{prop}
\begin{proof} The proof is similar to that of Proposition \ref{numwalls1}, but there are more cases to handle. As before, we can restrict to $r>0$; the two values for $\beta$ we get from $R>R_0$ are $\beta=-1/3+1/3=0$ and $\beta=-1/3-1/3=-2/3$; Bertram's Lemma and Lemma \ref{bounds} yield
\[0 < d < 6 \quad \text{and} \quad  -\dfrac{8}{3}r < d < 7 -\dfrac{8}{3}r,
\]
so that we get a contradiction only for $r\geq 3$; if $r=2$ we can have $d=1$, and if $r=1$ we can have $d=1,2,3$ or $4$.

Let's start with $r=2$ and $d=1$: since $R>R_0$, from (\ref{radius3}) we get $\dfrac{1}{12}+\dfrac{c}{4}>0$, i.e. $c> -\dfrac{1}{3}$; on the other hand, using Bogomolov Inequality on the potential quotient $Q$ we obtain $25+16(-2-c)\geq 0$ which means $c\leq -\dfrac{7}{16}$, a contradiction.

Now let $r=1$; again by $R>R_0$ we get $c > -\dfrac{d}{3}$, while $\Delta(Q)\geq 0$ yields $c\leq \dfrac{d^2-16}{8}$, which gives a contradiction only when $d=3$ or $d=4$, since $c$ must be an integer. 
When $d=1$ or $d=2$, instead, if we add in the Bogomolov Inequality for $F$ itself we get $c \leq \dfrac{d^2}{8}$, which implies $c=0$ in both cases.\end{proof}

We now have the a Theorem on walls for $\M(0,(2,2),4m+2)$ as well; see also Figure \ref{fig3} below.
\begin{thm}\label{main3} Let $v=(0, (2,2),-2)$ and $H=(1,2)$; there are three walls and four chambers for $\M_\sigma(v)$ on $\P^1\times\P^1$ in the $(\alpha,\beta)$-plane. From the outermost to the innermost, the walls are given by the following pairs of objects:
\begin{itemize}
\item[(i)] $\mathcal{W}''_2: \O(1,0), \ \O(-1,-2)[1]$;
\item[(ii)] $\mathcal{W}''_1: \O(0,1), \ \O(-2,-1)[1]$;
\item[(iii)] $\mathcal{W}''_0: \O^2, \ \O(-1,-1)^2[1]$.
\end{itemize}
In the same order, the moduli spaces corresponding to each chamber are given by
\begin{itemize}
\item[(i)] the Simpson moduli space $\S=\M(0, (2,2), 4m+2)$;
\item[(ii)] a projective variety $\S'$ obtained from $\S$ by contracting the divisor $\S_2$ to a (smooth) point;
\item[(iii)] a projective variety $\S''$ isomorphic to a GIT quotient, obtained from $\S'$ by contracting the divisor $\S_1$ to a (smooth) point;
\item[(iv)] the empty set $\emptyset$.
\end{itemize}
\end{thm}
\begin{proof}
By Theorem \ref{lvl} we know there are finitely many walls and in the outermost chamber we have $\M_\sigma(v)=\S$; now from (\ref{radius3}) we know that $\mathcal{W}''_2$ (resp.  $\mathcal{W}''_1$) is all lying on the left of $\beta=\dfrac{1}{2}$ (resp. $\beta=\dfrac{1}{2}$), so that in virtue of Proposition \ref{numwalls3} and Lemma \ref{mainlemma} we can replicate the exact same argument from Theorem \ref{main1}.

In fact, the outermost wall must be given by a sheaf subobject $F$; from the resolutions given in Theorem \ref{maicanstrata3} we compute $\Hom(\O(0,2),E)=0$ for all $E\in \S$ (see Appendix), so that $F\simeq\O(1,0)$. Again $\Hom(\O(1,0),E)=0$ unless $E\in \S_2$, hence $\mathcal{W}''_2$ is an actual wall for the stratum $\S_2$, with $Q\simeq \O(-1,-2)$[1]; one immediately computes $\Ext^1(\O(1,0),\O(-1,-2)[1])=\Ext^2(\O(1,0),\O(-1,-2)) = \Hom(\O(-1,-2),\O(-1,-2))=\C$, so that crossing the wall contracts the divisor $\S_2$ to a point; from \cite{moon1} we know it is a smooth point and the map is a blow-down.

Since there are no maps between $\O(1,0)$ and $\O(0,1)$ we know the trivial extension obtained by crossing the first wall is not involved in the second wall-crossing; the situation at the second wall is now completely symmetrical to the previous one and we have that crossing $\mathcal{W}''_1$ corresponds to contracting the divisor $\S_1$ to a smooth point; again by \cite{moon1}, the variety thus obtained is a GIT quotient isomorphic to the moduli space of semistable sheaves on $\P^3$ with Hilbert polynomial $m^2+3m+2$.

Finally, $\mathcal{W}''_0$ is a collapsing wall for $\S''$: in fact, by Theorem \ref{maicanstrata3} we know that $\O^2$ is a destabilizing subobject for all the sheaves in the open stratum and also $Q\simeq \O(-1,-1)^2[1]$; once we cross the wall we have $\Ext^1(\O,\O(-1,-1)[1])=\Hom(\O(-1,-1),\O(-2,-2))=0$  and this is enough to say that the open stratum vanishes and we get the empty set.

\end{proof}
\vspace{-0.5cm}
\begin{figure}[htbp]
\centering
\includegraphics[scale=0.25]{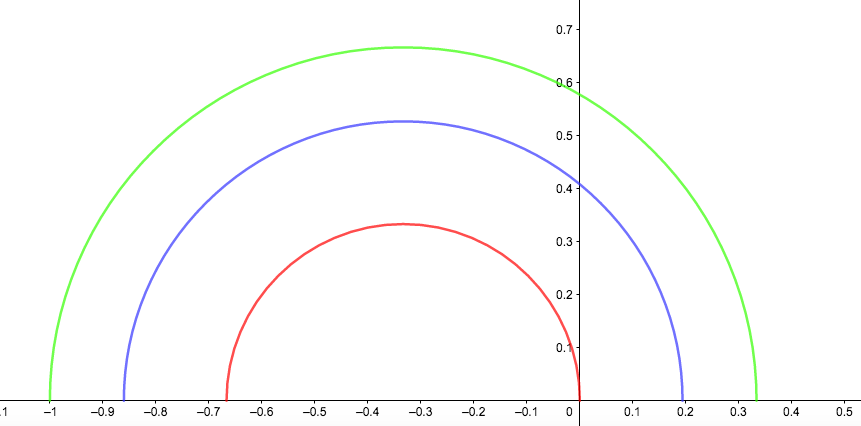}
\caption{The walls for $\M(0, (2,2), 4m+2)$ \label{fig3}}
\end{figure}
\vspace{-0.25cm}

\section{Appendix}
The computations from Theorems \ref{main1}, \ref{main2} and \ref{main3} are all very similar, so we only carry out the ones needed in Theorem \ref{main1}, which contain all the main ideas.

\subsection{Computations for Theorem \ref{main1}}
Recall $Q\simeq[\O(-2,-1)\oplus\O(-1,-2)\to\O(-1,-1)]$; the following computations were claimed throughout Theorem \ref{main1}:
\begin{itemize}
\item[(i)] $\Hom(\O(0,2),E)=0$ for all $E\in \M$;
\item[(ii)] $\Hom(\O(1,0),E)=0$ for all $E\in \M \setminus \M_2$;
\item[(iii)] $\Hom(\O(0,1),E)=0$ for all $E\in \M_0$;
\item[(iv)] $\Ext^1(Q,\O(0,1))=\C^{11}$;
\item[(v)] $\Ext^1(\O(0,1),Q)=\C$;
\item[(vi)] $\Ext^1(\mathcal{I}_{p,q}(1,1),\O(-1,-2)[1])=0$.
\end{itemize}
\begin{proof}[Proof of (i),(ii),(iii)] The vanishings of these groups are all proven in the same way, so we will only show the first one.

If $E\in \M_0$, then $\Hom(\O(0,2),E)$ fits into the long exact sequence
\small
\[
\dots\to \Hom(\O(0,2), \O^2) \to \Hom(\O(0,2),E) \to \Ext^1(\O(0,2),\O(-1,-2)\oplus \O(-1,-1))\to \dots,
\]
\normalsize
and clearly $\Hom(\O(0,2), \O^2)=0$; also $\Ext^1(\O(0,2),\O(-1,-2)) = \Ext^1(\O,\O(-1,-4))=H^1(\O(-1,-4))=0$ by K\"{u}nneth formula and similarly for $\Ext^1(\O(0,2),\O(-1,-2))$, hence the claim for $E$.
\end{proof}
\begin{proof}[Proof of (iv)] $\Hom(Q,\O(0,1))$ fits into the long exact sequence
\small
\begin{align*}
0&\to\Hom(Q,\O(0,1))\to \Hom(\O(-1,-1),\O(0,1)) \to \Hom (\O(-2,-1)\oplus\O(-1,-2),\O(0,1)) \to \\ &\to \Ext^1(Q,\O(0,1))\to \Ext^1(\O(-1,-1),\O(0,1)) \to \Ext^1(\O(-2,-1)\oplus\O(-1,-2),\O(0,1)) \to \dots;
\end{align*}
\normalsize
now $\Hom(Q,\O(0,1))=0$ by definition because $Q\in \mathcal{F}[1]$ and $\O(0,1)\in \mathcal{T}$, while we have $\Ext^1(\O(-1,-1),\O(0,1))= \Ext^1(\O,\O(1,2)) = H^1(\O(1,2))=0$ thanks to K\"{u}nneth formula.

One then easily computes $\Hom(\O(-1,-1),\O(0,1)) = \C^6$, $ \Hom (\O(-2,-1),\O(0,1))=\C^9$ and $ \Hom (\O(-1,-2),\O(0,1))=\C^8$, which altogether yield $\Ext^1(Q,\O(0,1))=\C^{11}$.
\end{proof}

\begin{proof}[Proof of (v)] Using the same complex, we now apply the functor $\Hom(\O(0,1), \text{ --- })$ to get
\small
\begin{align*}
\dots \to \Ext^1(\O(0,1),&\O(-1,-1)) \to \Ext^1(\O(0,1),Q) \to\\ &\to \Ext^2(\O(0,1),\O(-2,-1)\oplus\O(-1,-2)) \to \Ext^2(\O(0,1),\O(-1,-1))\to \dots;
\end{align*}
\normalsize
now $\Ext^1(\O(0,1),\O(-1,-1))=\Ext^1(\O,\O(-1,-2))= H^1(\O(-1,-2))=0$ by K\"{u}nneth again, and also $\Ext^2(\O(0,1),\O(-1,-1))= \Hom (\O(-1,-1),\O(-2,-1))=0$.

Moreover we have that $\Ext^2(\O(0,1),\O(-2,-1)\oplus\O(-1,-2)) = \Hom (\O(-2,-1)\oplus\O(-1,-2), \O(-2,-1))=\C$ so that also $\Ext^1(\O(0,1),Q)=\C$.
\end{proof}

\begin{proof}[Proof of (vi)] Let $F=\mathcal{I}_{p,q}(1,1) $; first of all $\Ext^1(F,\O(-1,-2)[1])=\Ext^2(F,\O(-1,-2))=\Hom (\O(-1,-2),F(-2,-2)) = \Hom (\O(1,0),F)$. Now we're going to use the fact that we have $F \simeq [\O(-1,-1)\to\O^2]$, and applying the functor $\Hom(\O(1,0), \text{ --- })$ we get
\[
\dots \to \Hom(\O(1,0),\O^2) \to \Hom (\O(1,0), F) \to \Ext^1(\O(1,0), \O(-1,-1)) \to \dots;
\]
as before $\Hom(\O(1,0),\O^2)=0$ easily, while also we have $\Ext^1(\O(1,0), \O(-1,-1))= H^1(\O(-2,-1))=0$ again by K\"{u}nneth, hence $\Hom (\O(1,0), F)=0$.
\end{proof}


\bibliographystyle{alpha}
\bibliography{biblio}

\begin{thebibliography}{ABCH13}

\bibitem[ABCH13]{p2}
Daniele Arcara, Aaron Bertram, Izzet Coskun, and Jack Huizenga.
\newblock The minimal model program for the hilbert scheme of points on
  $\mathbb{P}^2$ and {B}ridgeland stability.
\newblock {\em Advances in Mathematics}, 235:580--626, 3 2013.

\bibitem[AM16]{arcara2}
Daniele Arcara and Eric Miles.
\newblock Bridgeland stability of line bundles on surfaces.
\newblock {\em Journal of Pure and Applied Algebra}, 220:1655--1677, 2016.

\bibitem[AM17]{arcara}
Daniele Arcara and Eric Miles.
\newblock Projectivity of {B}ridgeland moduli spaces on {D}el {P}ezzo surfaces
  of {P}icard rank 2.
\newblock {\em International Mathematics Research Notices},
  2017(11):3426--3462, 2017.

\bibitem[BC13]{coskunbertram}
Aaron Bertram and Izzet Coskun.
\newblock {\em The birational geometry of the {H}ilbert scheme of points on
  surfaces}, pages 15--54.
\newblock Springer New York, 1 2013.

\bibitem[BM14]{projk3}
Arend Bayer and Emanuele Macr{\`i}.
\newblock Projectivity and birational geometry of {B}ridgeland moduli spaces.
\newblock {\em J. Amer. Math. Soc.}, 27, 2014.

\bibitem[BMS16]{abelian3}
Arend Bayer, Emanuele Macr{\`i}, and Paolo Stellari.
\newblock The space of stability conditions on abelian threefolds, and on some
  {C}alabi-{Y}au threefolds.
\newblock {\em Inventiones mathematicae}, 206(3):869--933, December 2016.

\bibitem[BMSZ17]{fano3}
Marcello Bernardara, Emanuele Macr{\`i}, Benjamin Schmidt, and Xiaolei Zhao.
\newblock Bridgeland stability conditions on {F}ano threefolds.
\newblock {\em {\'E}pijournal de G{\'e}om{\'e}trie Alg{\'e}brique}, 1,
  September 2017.

\bibitem[Bri07]{bridgeland1}
Tom Bridgeland.
\newblock Stability conditions on triangulated categories.
\newblock {\em Annals of Mathematics}, 166:317--345, 2007.

\bibitem[Bri08]{bridgeland2}
Tom Bridgeland.
\newblock Stability conditions on ${K}3$ surfaces.
\newblock {\em Duke Math. J.}, 141(2):241--291, February 2008.

\bibitem[CM14]{moon1}
Kiryong Chung and Han-Bom Moon.
\newblock Moduli of sheaves, {F}ourier-{M}ukai transform, and partial
  desingularization.
\newblock {\em Mathematische Zeitschrift}, 283, October 2014.

\bibitem[CM17]{moon2}
Kiryong Chung and Han-Bom Moon.
\newblock Birational geometry of the moduli space of pure sheaves on quadric
  surface.
\newblock {\em Comptes Rendus Mathematique}, March 2017.

\bibitem[GLHS16]{ben2}
Patricio Gallardo, C{\'e}sar Lozano~Huerta, and Benjamin Schmidt.
\newblock Families of elliptic curves in $\mathbb{P}^3$ and {B}ridgeland
  stability.
\newblock \url{https://arxiv.org/abs/1609.08184}, 2016.
\newblock To appear in Michigan Mathematical Journal.

\bibitem[HRS96]{happelsmalo}
Dieter Happel, Idun Reiten, and Sverre Smal{\o}.
\newblock Tilting in abelian categories and quasitilted algebras.
\newblock {\em Memoirs of the American Mathematical Society}, 575, 03 1996.

\bibitem[Kap88]{kapranov}
M.M. Kapranov.
\newblock On the derived categories of coherent sheaves on some homogeneous
  spaces.
\newblock {\em Inventiones mathematicae}, 92(3):479--508, 1988.

\bibitem[Li16]{li}
Chunyi Li.
\newblock Stability conditions on {F}ano threefolds of {P}icard number one.
\newblock \url{https://arxiv.org/abs/1510.04089}, March 2016.
\newblock To appear in Journal of the European Mathematical Society.

\bibitem[LP93]{lepot}
Joseph Le~Potier.
\newblock Faisceaux semi-stables de dimension 1 sur le plan projectif.
\newblock {\em Rev. Roumaine Math. Pures Appl.}, 38:635--678, 1993.

\bibitem[LZ18]{zhaop2}
Chunyi Li and Xiaolei Zhao.
\newblock Birational models of moduli spaces of coherent sheaves on the
  projective plane.
\newblock \url{https://arxiv.org/abs/1603.05035}, 2018.
\newblock To appear in Geometry and Topology.

\bibitem[Mac07]{macri1}
Emanuele Macr{\`i}.
\newblock Stability conditions on curves.
\newblock {\em Math. Res. Lett.}, 14:657--672, 2007.

\bibitem[Mai17]{maican1}
Mario Maican.
\newblock On the geometry of the moduli space of sheaves supported on curves of
  genus two in a quadric surface.
\newblock \url{https://arxiv.org/abs/1706.00876}, June 2017.

\bibitem[Mai18]{maican2}
Mario Maican.
\newblock Moduli of sheaves supported on curves of genus two in a quadric
  surface.
\newblock {\em Geometriae Dedicata}, April 2018.

\bibitem[MS17]{macri2}
Emanuele Macr{\`i} and Benjamin Schmidt.
\newblock Lectures on {B}ridgeland {S}tability.
\newblock \url{https://arxiv.org/abs/1607.01262}, March 2017.

\bibitem[Neu16]{enriques}
Howard Neuer.
\newblock Projectivity and birational geometry of {B}ridgeland moduli spaces on
  an {E}nriques surface.
\newblock {\em Proceedings of the London Mathematical Society}, 113, September
  2016.

\bibitem[Sch15]{ben1}
Benjamin Schmidt.
\newblock Bridgeland stability on threefolds - some wall crossings.
\newblock \url{https://arxiv.org/abs/1509.04608}, 2015.

\bibitem[Sim94]{simpson}
Carlos~T. Simpson.
\newblock Moduli of representations of the fundamental group of a smooth
  projective variety i.
\newblock {\em Inst. Hautes {\'E}tudes Sci. Publ. Math}, 79:47--129, 1994.

\end{thebibliography}

\end{document}